\documentclass[11pt]{article}
\usepackage{amsfonts,amsmath,amsthm,amssymb}
\usepackage{graphicx}
\theoremstyle{plain}
\newtheorem{theorem}{Theorem}[section]

\newtheorem{corollary}{Corollary}[section]
\newtheorem{example}{Example}
\newtheorem{definition}{Definition}[section]
\newtheorem{lemma}{Lemma}[section]
\begin{document}

\title{\textbf{Constructing isothermal curvature line coordinates on surfaces which admit them}}
\author
{Eugenio Aulisa (eugenio.aulisa@ttu.edu)\\
Department of Mathematics and Statistics\\
Texas Tech University, Broadway\&Boston Ave.
Lubbock, TX  79409-1042, USA\\
Magdalena Toda(magda.toda@ttu.edu)\\
Department of Mathematics and Statistics\\
Texas Tech University, Broadway\&Boston Ave.
Lubbock, TX  79409-1042, USA\\
Zeynep Sultan Kose (zsultankose@hotmail.com)\\
Department of Mathematics and Statistics\\
Texas Tech University, Broadway\&Boston Ave.
Lubbock, TX  79409-1042, USA\\
}

\date{}
\maketitle
\numberwithin{equation}{section}
\thispagestyle{empty}
\begin{abstract}
{\it Isothermic parameterizations} are synonyms of isothermal curvature line parameterizations, for surfaces immersed in Euclidean spaces. We provide a method of constructing isothermic coordinate charts on surfaces which admit them, starting from an arbitrary chart. One of the primary applications of this work consists of numerical algorithms for surface visualization. 
\end{abstract}


\newpage

\section
{ Constructing Isothermic Coordinates (Curvature Line Isothermal Charts) }

{\it Isothermic parameterizations} are synonyms of isothermal curvature line parameterizations \cite{DO,lE,AP}, for surfaces immersed in Euclidean spaces. By definition, {\it isothermic surfaces} represent immersed surfaces which admit at least one isothermic parameterization. The initial motivation of this project came forward from solving Bonnet problems analytically and numerically, by using Cartan's theory of structure equations, and numerical analysis. The authors also studied dual surfaces, in the sense of Christoffel transforms. Such transforms exist exclusively for isothermic coordinate charts.

In a previous work \cite{KTA}, we solved Bonnet problems in terms of Cartan theory in conformal coordinates in the most general context of regular surfaces. That is, for any compatible first and second fundamental form, corresponding families of solutions were obtained (up to roto-translations).

Due to the high impact of isothermic coordinates in various areas, it became important to study necessary and sufficient conditions for the existence of isothermic coordinates on surfaces; and whenever such coordinates exist, a practical method of constructing them became imperative.

We provide a method of constructing isothermic coordinate charts on surfaces which admit them, starting from an arbitrary chart (parameterization). One of the primary applications of this work consists of numerical algorithms for surface visualization. Many of these results can be generalized to immersions of $2$-dimensional manifolds in n-dimensional space forms, but our main goal was a construction algorithm with numerical implementation and visualization.

\section{Fundamental Concepts of Surface Theory}

This work exclusively involves immersed (respectively, embedded) surfaces in the Euclidean $3$-dimensional space.

\begin{definition}
\label{surf}

A {\it local surface} is a differentiable mapping from an open, simply connected subset of the Euclidean plane, to $\mathbb{R}^3$. Various references use the name of {\it patch of a surface}, or {\it local patch} instead of local surface. We will call such a local surface a {\it surface immersion}, if the Jacobian of the differential mapping has rank two at every point. This is equivalent to the differential map being 1-1. Remark that such a surface is generally not embedded. If the local surface (immersion) represents an injective mapping in itself, then we will call such a mapping an {\it embedding}, or an {\it injective patch}.

In a more global sense, a connected set $M$ in $\mathbb{R}^3$ is said to be a {\it regular 2-dimensional surface} if for an arbitrary point $p\in M$ there exist an open ball $U_p$ in $\mathbb{R}^3$ with center at $p$ and a homeomorphism  $\psi $ which maps $M\cap U_p$ onto an open disk on some plane in the space $\mathbb{R}^3$. Such a homeomorphism is frequently called a {\it chart}, or a {\it system of local coordinates on a surface}.
\end{definition}

It is worth noting that, in the previous definition, the word disk can be replaced by an arbitrary open set of a plane, diffeomorphic to a disk.
\

Let us consider an immersion $r=r(u,v)=(X(u,v),Y(u,v),Z(u,v))$ (of an open, simply connected set) in $\mathbb{R}^3$. Note that the property of this map $r=r(u,v)$ being an immersion is equivalent to the property that the vectors $\frac{\partial{r}}{\partial{u}}$ and $\frac{\partial{r}}{\partial{v}}$ are linearly independent at every point. These span a tangent plane at each point, that is, generate a tangent bundle. We will denote $N=\frac{r_u\times r_v}{|r_u\times r_v|}$ and call it the unit normal vector field, or Gauss map; it represents a map $N:D\rightarrow{S^2}$.

\begin{definition}
We will call the following quadratic form
\begin{equation*}
ds^{2}(u,v)=<r_u,r_u>du^{2}+2<r_{u},r_{v}>dudv+<r_v,r_v>dv^{2}
\end{equation*}
a {\it naturally induced metric}, or {\it first fundamental form}.
\end{definition}

Recall that the notion of induced metric came from a bilinear map that acts as follows: $I (v_p, w_p) = <v_p, w_p>$, where the vectors $v_p$ and $w_p$ represent tangent vectors to the surface at the point $p$.
Also recall that, locally, every smooth surface can be represented as a graph (Monge chart). In general, we will not work with Monge charts, but with generic immersions.

If we define
$E=|r_u|^2$, $F=<r_u,r_v>$, $G=|r_v|^2$, then we can identify the naturally induced metric (2-form)
\begin{equation*}
 ds^2=E du^2+2F dudv+ G dv^2, 
\end{equation*}
with a symmetric and positive definite real-valued matrix.

As it is very well known, the notations above are classical, and were first used by Gauss.

\begin{definition} For a local immersion in $\mathbb{R}^3$, with image $M=r(D)$, let $N$ represent the Gauss map. For a tangent vector $v=v_p$ to $M$ at $p$, we put

$S_p(v) = - dN_p(v) = - D_v N(p)$.

Then $S_p$, viewed as a linear map from the tangent plane to itself, defines the shape operator $S$ at every point.

We define the normal curvature as

$$\kappa_n (v_p) = \frac{<S(v_p), v_p>}{<v_p, v_p>}.$$

Clearly, if $v_p$ is unitary, then the denominator is equal to 1.

\end{definition}

Note that, if the actual immersion a surface is unknown and only the first fundamental form is given, then we do not possess much information about the geometric properties of the surface. By a classical result of Bonnet, the geometry of surfaces exclusively depends on two quadratic differential forms. One of them is the metric, and the other is the second fundamental form.

\begin{definition}

The second fundamental form of a regular surface $M$ in $\mathbb{R}^3$ represents a symmetric and bilinear form on the tangent plane to $M$ at the point $p$ given by:

$II(v_p, w_p) = < S(v_p), w_p >$
for any arbitrary tangent vectors on the tangent plane at $p$.

\end{definition}

Another classical notation of the second fundamental form that is due to Gauss is:
\begin{equation*}
II: =ldu^2+2mdudv+ndv^2,
\end{equation*}
where the coefficients $l$, $m$, $n$ are given by: $l=<r_{uu}, N>$, $m=<r_{uv}, N>$, $n=<r_{vv}, N>$.

Unlike the first fundamental form, the second fundamental form is not necessarily positive definite. However, it is a symmetric bilinear form. Since all directions in a tangent plane $T_pM$ form a compact set homeomorphic to a circle, $\kappa$ has at least one minimum and one maximum, i.e., at least two extremal values. The maximum normal curvature $\kappa_1$ and the minimum normal curvature $\kappa_2$ are called principal curvatures. They represent the eigenvalues of the shape operator, and the eigenvectors corresponding to them are called principal directions.

\begin{definition}
 The average and the product of the two principal curvatures are called mean curvature and Gaussian curvature, respectively.
\end{definition}

In the classical literature, surfaces whose mean curvature (respectively, Gaussian curvature) is a constant have been studied extensively. Constant mean curvature (CMC) surfaces represent a special subclass of Bonnet surfaces, which in their turn represent an important subclass of isothermic surfaces.

\begin{definition}
If the induced metric is a (nonconstant) multiple of the flat metric, then the immersion is said to be {\it conformal}, or {\it isothermal}, or {\it in conformal coordinates}, i.e. $\left\|r_{u}\right\|=\left\|r_{v}\right\|$ and $<r_u,r_v>=0$ at every point.

\

A surface is said to be (parameterized) {\it in curvature line coordinates}, if both first and second fundamental forms are diagonal.

\

A surface in $\mathbb{R}^3$ is said to be parameterized in {\it isothermic coordinates} if the given parameterization is conformal and in curvature line coordinates at the same time, i.e. it is given as an immersion $f:D\subset{R^2}\longrightarrow{R^3}$ such that $$<f_{u},f_{v}>=0,$$ and
\begin{equation*}
\left\|f_{u}\right\|^{2}=\left\|f_v\right\|^{2}=E(u,v),
\end{equation*}

for a smooth function over $D$, and $<f_{uv},N>=0$.

\end{definition}

Remark that obtaining a parameterization that is both isothermal and in curvature line coordinates at the same time is not possible on {\bf most} regular surfaces.

Also note that at umbilic points (i.e., points where the normal curvature is the same in every direction), the actual choice of an orthogonal system of eigenvectors becomes problematic. Therefore, we usually view all the umbilics as singularities. Clearly, this point of view makes the sphere a very special example of isothermic surface!

Assume that we are dealing with a regular surface in $\mathbb R^3$ which admits isothermic coordinates (away from a discrete set of singularities). We will call such a surface an {\it isothermic surface}.

Isothermic surfaces were extensively studied in \cite{Chen}, where the main interest was to study Bonnet surfaces as a particular case of isothermic surfaces.

It is important to note that \cite{Chen} mentioned the following important characteristic property of isothermic surfaces:

{\it $M$ is an isothermic surface in $\mathbb R^3$ if and only if locally there exists a conformal parameter $
z=u+iv$ such that the Hopf differential coefficient, $Q=Q(z,\bar z)$ is a real-valued function.}

This characteristic property (in terms of $Q$ being real) is obvious from the definition of an isothermic surface, together with the formula of $Q$ in terms of $l$, $m$ and $n$.

A more comprehensive characterization of isothermic surfaces was provided by Bobenko in \cite{BO}:

\begin{lemma}
Let $f:D\longrightarrow{R^3}$ be a conformal immersion of an umbilic-free surface. The surface is isothermic if and only if there exists a holomorphic non-vanishing differential $f(z) dz^2$ on $D$ and a function $q$ from $D$ to $R_+$ such that the Hopf differential is of the form:
       $ Q(z,\bar z)= f(z) q(z,\bar z).$
It is easy to see that, in this case, $w = \int {\sqrt {f(z)}} dz$ is an isothermic coordinate.

\end{lemma}

We arrived at the following question, which is important to answer in this work, from both a theoretical and a constructive (applicative) point of view:

{\it Question: ``Starting from a arbitrarily given immersion  $(x,y)\mapsto f(x,y)$ of an isothermic surface in the Euclidean 3-space, is there a simple method of obtaining an isothermic parameterization $(\beta, \gamma)\mapsto f(\beta, \gamma)$ corresponding to it?''}

\

We implemented a key idea for changing the old coordinates to the isothermic ones, which can be summarized in two steps:

\

a). applying the Gram-Schmidt orthogonalization;

and

b). multiplying the velocity vectors by a `non-constant scaling function' $K=K(x,y)$, namely a class $C^2$ non-constant function specific to this chapter (not to be confused with the Gaussian curvature).

Corresponding to this method, we obtained a condition that this function, $K$, must satisfy, and we constructed isothermic coordinates from an arbitrary immersion, under the above-stated condition.

Solving a Bonnet problem gives rise to a generic immersion, which is rarely isothermic; constructing an isothermic parameterization for the same physical surface, whenever possible, presents a great deal of simplification.

An isothermic parameterization creates a special mesh that is desirable in some applications, such as:

\begin{itemize}

\item applying geometric methods to mesh discretization which are both homogeneous and orthogonal;
\item in structural mechanics, deriving the equation of shells along the principal directions;
\item in architecture/industrial design, working with orthogonal meshes is often preferrable, in order to minimize the production cost;
\item finding dual surfaces (Christoffel transforms).
\end{itemize}

\newpage

\begin{theorem}
Let $f(x,y)=( f^1(x,y),f^2(x,y),f^3(x,y))$ be an {\bf arbitrary} immersion, with velocities $f_x$ and $f_y$, of a regular surface $M=f(D)$, where $D$ represents an open and simply connected domain. Let us define the vector distributions
\begin{eqnarray}
&& f_1:=K(x,y)\left(\frac{f_x}{\sqrt{E}}\cos(\alpha)+\frac{{E}f_y-{F}f_x}{{\sqrt {E}}{\sqrt{ EG-F^2}}}\sin(\alpha)\right), \label{ffgambet} \\
&& f_2:=K(x,y)\left(-\frac{f_x}{\sqrt{E}}\sin(\alpha)+\frac{{E}f_y-{F}f_x}{{\sqrt {E}}{\sqrt{ EG-F^2}}}\cos(\alpha)\right),
\label{ffgambeta}
\end{eqnarray}
where the function $K(x,y)$ is positive and smooth, and the angle $\alpha$ satisfies the following equation (in terms of coefficients of the first and second fundamental forms corresponding to f):

\begin{equation}
\tan(2\alpha)=-2\left( {\frac{(-F l+ Em){\sqrt{EG-F^2}}} {(2{F^2}-E G)l-2 E F m +{E^2} n}}\right). \label{alphann}
\end{equation}

 Further assume that there exists a change of coordinate map $(x,y)\mapsto(\beta,\gamma)$ such that the given vectors $f_1$ and $f_2$ respectively represent the partial velocities $f_{\gamma}$, $f_{\beta}$ of $f$.

{\bf In these assumptions, the new coordinates $(\beta, \gamma)$ are isothermic {\it if and only if} the corresponding smooth function $K$, with respect to the two different charts, will satisfy the following equations:}

\begin{eqnarray}
&& K_\gamma=K^2\left( \frac{\frac{F^2}{E^2}E_x-2\frac{F}{E}F_x+G_x}{2(G-\frac{F^2}{E})\sqrt{E}}\cos(\alpha)-\frac{-E_y-\frac{F}{E}E_x+2F_x}{2E\sqrt{G-\frac{F^2}{E}}}\sin(\alpha) \right) \nonumber\\
&& \quad +K^2\left(-\frac{\alpha_x}{\sqrt{E}}\sin(\alpha)+\frac{\alpha_y-\frac{F}{E}\alpha_x}{\sqrt{G-\frac{F^2}{E}}}\cos(\alpha) \right),  \label{condk}\\
&& K_\beta=K^2\left( -\frac{\frac{F^2}{E^2}E_x-2\frac{F}{E}F_x+G_x}{2(G-\frac{F^2}{E})\sqrt{E}}\sin(\alpha)-\frac{-E_y-\frac{F}{E}E_x+2F_x}{2E\sqrt{G-\frac{F^2}{E}}}\cos(\alpha) \right) \nonumber \\
&&\quad +K^2\left(-\frac{\alpha_x}{\sqrt{E}}\cos(\alpha) -\frac{\alpha_y-\frac{F}{E}\alpha_x}{\sqrt{G-\frac{F^2}{E}}}\sin(\alpha) \right).  \label{condka}
\end{eqnarray}

\end{theorem}

\newpage

{\bf Remarks} about the previous theorem:

A). Note that the geometric meaning of $\alpha(x,y)$ is of a rotation angle that changes at every point (of a pre-orthonormalized frame after Gramm-Schmitt); and the geometric meaning of $K(x,y)$ is of a non-constant dilating (scaling) function of the unit tangent vectors, up to a pair of partial velocity vectors in isothermic coordinates at each point on the surface (both of magnitude $K(x,y)$).

B). Consider the fractional expression of $\tan (2\alpha)$. If the initial parameterization would be an isothermic one, one can easily check that both the numerator and the denominator would become {\bf zero} at umbilic points. For umbilics, the expression would then be undetermined, and not undefined. 
In the following  statement we assume that we are always away from umbilic points. 
However, even for the case of isolated umbilic points, the discrete algorithm allows a certain flexibility in choosing the size steps in such a way that the umbilic points can be avoided.

C). The formulas above provide a simple method to reparameterize to isothermic coordinates. We are not claiming that this is the only method.

D). In this theorem, equations \eqref{ffgambet}-\eqref{ffgambeta} and \eqref{condk}-\eqref{condk} assure the compatibility condition $(f_1)_\beta=(f_2)_\gamma$, which will become: $(f_{\gamma})_\beta=(f_{\beta})_\gamma$.

\begin{proof}

Note that our long expressions in the proof will not be simplified to the maximum extent, due to our willingness to show all the computations exactly as they were successively performed.

By the definition of \textbf{isothermic} coordinates, the vectors $f_1=f_{\gamma}$ and $f_2=f_{\beta}$ represent the partial velocity vector fields of an isothermic parametrization $(\beta, \gamma)$ if and only if they satisfy the following conditions simultaneously:
\begin{enumerate}
 \item $||f_{\gamma}||=||f_{\beta}||$,
 \item $<f_{\gamma},f_{\beta}>=0$,
 \item $<f_{\gamma \beta},N>=0$.
 \item $(f_{\gamma})_\beta=(f_{\beta})_\gamma.$
\end{enumerate}
\begin{enumerate}
 \item It is easy to rewrite the first condition in an equivalent way. Specifically, we obtain
{\footnotesize
\begin{eqnarray*}
&& ||f_{\gamma}||^2=<f_{\gamma},f_{\gamma}>=K^2\left( \frac{||f_x||^2}{E}\cos^2(\alpha)+2\frac{<f_x,f_y>-\frac{F}{E}||f_x||^2}{\sqrt{E}\sqrt{G-\frac{F^2}{E}}}\cos(\alpha)\sin(\alpha) \right) \\
&& \quad +K^2 \left( \frac{||f_y||^2+\frac{F^2}{E^2}||f_x||^2}{G-\frac{F^2}{E}}\sin^2(\alpha) \right) \\
&& =K^2,
\end{eqnarray*}
}
and
{\footnotesize
\begin{eqnarray*}
&& ||f_{\beta}||^2=<f_{\beta},f_{\beta}>=K^2\left( \frac{||f_x||^2}{E}\sin^2(\alpha)-2\frac{<f_x,f_y>-\frac{F}{E}||f_x||^2}{\sqrt{E}\sqrt{G-\frac{F^2}{E}}}\cos(\alpha)\sin(\alpha) \right) \\
&& \quad +K^2 \left( \frac{||f_y||^2+\frac{F^2}{E^2}||f_x||^2}{G-\frac{F^2}{E}}\cos^2(\alpha) \right) \\
&&  =K^2.
\end{eqnarray*}
}

Note that the $||f_{\gamma}||=||f_{\beta}||=K$ is an immediate consequence of the definition of the initial vector field distribution $f_1, f_2$.

\item The second condition is also immediately implied by the two vector fields provided
{\footnotesize
\begin{eqnarray*}
&& <f_{\gamma},f_{\beta}>=K^2\left( -\frac{||f_x||^2}{E}\cos(\alpha)\sin(\alpha)+\frac{<f_x,f_y>-\frac{F}{E}||f_x||^2}{\sqrt{E}\sqrt{G-\frac{F^2}{E}}}\cos^2(\alpha) \right) \\
&& \;+K^2 \left(-\frac{<f_x,f_y>-\frac{F}{E}||f_x||^2}{\sqrt{E}\sqrt{G-\frac{F^2}{E}}}\sin^2(\alpha) \right) \\
&& \quad+K^2\left(\frac{||f_y||^2-2\frac{F}{E}<f_y,f_x>+\frac{F^2}{E^2}||f_x||^2}{G-\frac{F^2}{E}}\sin(\alpha)\cos(\alpha)\right) \\
&& =0.
\end{eqnarray*}
}

\item We will show that even the third condition is also implied by the vector construction in the theorem hypothesis.

 In order to reformulate the third condition, let  $D_1(x,y)=\frac{f_x}{\sqrt{E}}$ and $D_2(x,y)=\frac{f_y-\frac{F}{E}f_x}{\sqrt{G-\frac{F^2}{E}}}$, for simplification.

Therefore, $f_{\gamma}$ and $f_{\beta}$ will become
{\footnotesize
\begin{eqnarray}
&& f_\gamma:=K\left(D_1 \cos(\alpha)+D_2 \sin(\alpha)\right), \label{fgbd} \\
&&f_\beta:=K\left(-D_1 \sin(\alpha)+D_2 \cos(\alpha)\right).\label{fgbda}
\end{eqnarray}
}
For the last identity, consider
{\footnotesize
\begin{eqnarray}
&& <(f_{\gamma})_\beta,N>= <K_\beta(D_1\cos(\alpha)+D_2\sin(\alpha))+K^2 (-\frac{\frac{\partial }{\partial x}D_1}{\sqrt{E}}\sin(\alpha)\cos(\alpha) \nonumber\\
&&\; -\frac{\frac{\partial }{\partial x}D_2}{\sqrt{E}}\sin^2(\alpha)+\frac{\frac{\partial }{\partial y}D_1}{\sqrt{G-\frac{F^2}{E}}}\cos^2(\alpha)+\frac{\frac{\partial }{\partial y}D_2}{\sqrt{G-\frac{F^2}{E}}}\sin(\alpha)\cos(\alpha) \nonumber \\
&& \quad-\frac{F}{E}\frac{\frac{\partial }{\partial x}D_1}{\sqrt{G-\frac{F^2}{E}}}\cos^2(\alpha)+\frac{D_1 \alpha_x}{\sqrt{E}}\sin^2(\alpha)-\frac{F}{E}\frac{\frac{\partial }{\partial x}D_2}{\sqrt{G-\frac{F^2}{E}}}\sin(\alpha)\cos(\alpha) \nonumber \\
&& \quad\;-\frac{D_2 \alpha_x}{\sqrt{E}}\sin(\alpha)\cos(\alpha)-\frac{D_1 \alpha_y}{\sqrt{G-\frac{F^2}{E}}}\sin(\alpha)\cos(\alpha)+\frac{D_2 \alpha_y}{\sqrt{G-\frac{F^2}{E}}}\cos^2(\alpha)\nonumber \\
&& \qquad+\frac{F}{E}\frac{D_1 \alpha_x}{\sqrt{G-\frac{F^2}{E}}}\sin(\alpha)\cos(\alpha)-\frac{F}{E}\frac{D_2 \alpha_x}{\sqrt{G-\frac{F^2}{E}}}\cos^2(\alpha) ),N > . \label{fgbN}
\end{eqnarray}
}
Since $<f_x,N>=0$ and $<f_y,N>=0$, then all terms with $<D_1,N>$ and $<D_2,N>$ will be zero; replacing the partial derivatives of $D_1$ and $D_2$ into the equation \eqref{fgbN} we obtain
{\footnotesize
\begin{eqnarray}
&& <(f_{\gamma})_\beta,N>=\left(\frac{-\frac{F}{E}}{\sqrt{G-\frac{F^2}{E}}\sqrt{E}}(\cos^2(\alpha)-\sin^2(\alpha))\right)<f_{xx},N> \nonumber\\
&& \,+\left(\frac{2\frac{F^2}{E}-G}{E\left(G-\frac{F^2}{E}\right)}\sin(\alpha)\cos(\alpha) \right)<f_{xx},N> \nonumber \\
&& \;+\left(\frac{1}{\sqrt{G-\frac{F^2}{E}} \sqrt{E}}(\cos^2(\alpha)-\sin^2(\alpha))-2 \frac{\frac{F}{E}}{G-\frac{F^2}{E}}\sin(\alpha)\cos(\alpha) \right)<f_{xy},N> \nonumber \\
&& \;\,+\left( \frac{1}{G-\frac{F^2}{E}}\cos(\alpha)\sin(\alpha) \right)<f_{yy},N> \nonumber \\
&& =\left(\frac{-\frac{F}{E}}{\sqrt{G-\frac{F^2}{E}}\sqrt{E}}l+\frac{1}{\sqrt{G-\frac{F^2}{E}} \sqrt{E}}m\right)\cos(2\alpha) \nonumber \\
&&\; +\left(\frac{2\frac{F^2}{E}-G}{E\left(G-\frac{F^2}{E}\right)}l-2 \frac{\frac{F}{E}}{G-\frac{F^2}{E}}m+\frac{1}{G-\frac{F^2}\;{E}}n \right) \frac{\sin(2\alpha)}{2}. \label{fgambet}
\end{eqnarray}
}
Now, if we rewrite the equation \eqref{alphann} as
{\footnotesize
\begin{equation*}
\sin(2\alpha)=-2\left( \frac{\frac{-\frac{F}{E}}{\sqrt(G-\frac{F^2}{E})\sqrt{E}}l+\frac{1}{\sqrt{\left(G-\frac{F^2}{E}\right)} \sqrt{E}}m}{\frac{2\frac{F^2}{E}-G}{E\left(G-\frac{F^2}{E}\right)}l-2 \frac{\frac{F}{E}}{G-\frac{F^2}{E}}m+\frac{1}{G-\frac{F^2}{E}}n}\right)\cos(2\alpha),
\end{equation*}
}
and plug it into the equation \eqref{fgambet}, then we obtain
{\footnotesize
\begin{eqnarray*}
&& <(f_{\gamma})_\beta,N>=\left(\frac{-\frac{F}{E}}{\sqrt{G-\frac{F^2}{E}}\sqrt{E}}l+\frac{1}{\sqrt{G-\frac{F^2}{E}} \sqrt{E}}m\right)\cos(2\alpha) \nonumber \\
&&\; -\left(\frac{2\frac{F^2}{E}-G}{E\left(G-\frac{F^2}{E}\right)}l-2 \frac{\frac{F}{E}}{G-\frac{F^2}{E}}m+\frac{1}{G-\frac{F^2}{E}}n \right) \nonumber \\
&&\quad \left( \frac{\frac{-\frac{F}{E}}{\sqrt{G-\frac{F^2}{E}}\sqrt{E}}l+\frac{1}{\sqrt{G-\frac{F^2}{E}} \sqrt{E}}m}{\frac{2\frac{F^2}{E}-G}{E\left(G-\frac{F^2}{E}\right)}l-2 \frac{\frac{F}{E}}{G-\frac{F^2}{E}}m+\frac{1}{G-\frac{F^2}{E}}n}\right)\cos(2\alpha)=0. \nonumber \\
\end{eqnarray*}
}
\item Last but not least, we will analyze the fourth condition in its equivalent forms, and show it is actually equivalent to the PDE system satisfied by the function $K$. Again, consider the equations \eqref{fgbd}-\eqref{fgbda}

{\footnotesize
\begin{eqnarray*}
&& (f_{\gamma})_\beta= K_\beta \left( D_1\cos(\alpha)+D_2\sin(\alpha)\right)+K\frac{\partial}{\partial \beta}\left(D_1\cos(\alpha)+D_2\sin(\alpha)\right)   \\
&& = K_\beta(D_1\cos(\alpha)+D_2\sin(\alpha))+K ( -\frac{\frac{\partial}{\partial x}( D_1\cos(\alpha)+D_2\sin(\alpha))}{\sqrt{E}}\sin(\alpha) \nonumber \\
&& \;+\frac{\frac{\partial}{\partial y}(D_1\cos(\alpha)+D_2\sin(\alpha))-\frac{F}{E} \frac{\partial}{\partial x}(D_1\cos(\alpha)+D_2\sin(\alpha))}{\sqrt{G-\frac{F^2}{E}}}\cos(\alpha) )  \\
&& = K_\beta(D_1\cos(\alpha)+D_2\sin(\alpha))+K^2 (-\frac{\frac{\partial }{\partial x}D_1}{\sqrt{E}}\sin(\alpha)\cos(\alpha)-\frac{\frac{\partial }{\partial x}D_2}{\sqrt{E}}\sin^2(\alpha)  \\
&& \; +\frac{\frac{\partial }{\partial y}D_1}{\sqrt{G-\frac{F^2}{E}}}\cos^2(\alpha)+\frac{\frac{\partial }{\partial y}D_2}{\sqrt{G-\frac{F^2}{E}}}\sin(\alpha)\cos(\alpha)-\frac{F}{E}\frac{\frac{\partial }{\partial x}D_1}{\sqrt{G-\frac{F^2}{E}}}\cos^2(\alpha)  \\
&& \quad +\frac{D_1 \alpha_x}{\sqrt{E}}\sin^2(\alpha)-\frac{F}{E}\frac{\frac{\partial }{\partial x}D_2}{\sqrt{G-\frac{F^2}{E}}}\sin(\alpha)\cos(\alpha)-\frac{D_2 \alpha_x}{\sqrt{E}}\sin(\alpha)\cos(\alpha)  \\
&& \quad\;-\frac{D_1 \alpha_y}{\sqrt{G-\frac{F^2}{E}}}\sin(\alpha)\cos(\alpha)+\frac{D_2 \alpha_y}{\sqrt{G-\frac{F^2}{E}}}\cos^2(\alpha)\\
&& \qquad +\frac{F}{E}\frac{D_1 \alpha_x}{\sqrt{G-\frac{F^2}{E}}}\sin(\alpha)\cos(\alpha)-\frac{F}{E}\frac{D_2 \alpha_x}{\sqrt{G-\frac{F^2}{E}}}\cos^2(\alpha) ), 
\end{eqnarray*}
}
{\footnotesize
\begin{eqnarray*}
 && (f_{\beta})_\gamma= K_\beta(-D_1\sin(\alpha)+D_2\cos(\alpha))+K\frac{\partial}{\partial \beta}(-D_1\sin(\alpha)+D_2\cos(\alpha))  \\
&& = K_\beta(-D_1\sin(\alpha)+D_2\cos(\alpha))+K (\frac{\frac{\partial}{\partial x}(-D_1\sin(\alpha)+D_2\cos(\alpha))}{\sqrt{E}}\cos(\alpha)  \\
&&\; +\frac{\frac{\partial}{\partial y}(-D_1\sin(\alpha)+D_2\cos(\alpha))-\frac{F}{E} \frac{\partial}{\partial x}(-D_1\sin(\alpha)+D_2\cos(\alpha))}{\sqrt{G-\frac{F^2}{E}}}\sin(\alpha) )\\
&& = K_\gamma(-D_1\sin(\alpha)+D_2\cos(\alpha))+K^2 (-\frac{\frac{\partial }{\partial x}D_1}{\sqrt{E}}\sin(\alpha)\cos(\alpha)+\frac{\frac{\partial }{\partial x}D_2}{\sqrt{E}}\cos^2(\alpha)  \\
&& \;-\frac{\frac{\partial }{\partial y}D_1}{\sqrt{G-\frac{F^2}{E}}}\sin^2(\alpha)+\frac{\frac{\partial }{\partial y}D_2}{\sqrt{G-\frac{F^2}{E}}}\sin(\alpha)\cos(\alpha)+\frac{F}{E}\frac{\frac{\partial }{\partial x}D_1}{\sqrt{G-\frac{F^2}{E}}}\sin^2(\alpha)  \\
&& \quad +\frac{D_1 \alpha_x}{\sqrt{E}}\sin^2(\alpha)-\frac{F}{E}\frac{\frac{\partial }{\partial x}D_2}{\sqrt{G-\frac{F^2}{E}}}\sin(\alpha)\cos(\alpha)-\frac{D_2 \alpha_x}{\sqrt{E}}\sin(\alpha)\cos(\alpha)  \\
&& \quad\;-\frac{D_1 \alpha_y}{\sqrt{G-\frac{F^2}{E}}}\sin(\alpha)\cos(\alpha)-\frac{D_2 \alpha_y}{\sqrt{G-\frac{F^2}{E}}}\sin^2(\alpha) \\
&& \qquad +\frac{F}{E}\frac{D_1 \alpha_x}{\sqrt{G-\frac{F^2}{E}}}\sin(\alpha)\cos(\alpha)+\frac{F}{E}\frac{D_2 \alpha_x}{\sqrt{G-\frac{F^2}{E}}}\sin^2(\alpha) ).
\end{eqnarray*}
}
The compatibility condition is rewritten as
{\footnotesize
\begin{eqnarray*}
&& 0=(f_{\gamma})_\beta-(f_{\beta})_\gamma=K_\beta(D_1\cos(\alpha)+D_2\sin(\alpha))-K_\gamma(-D_1\sin(\alpha)+D_2\cos(\alpha))  \nonumber \\
&& \;+K^2(-\frac{\frac{\partial }{\partial x}D_2}{\sqrt{E}}+\frac{\frac{\partial }{\partial y}D_1}{\sqrt{G-\frac{F^2}{E}}}-\frac{F}{E}\frac{\frac{\partial }{\partial x}D_1}{\sqrt{G-\frac{F^2}{E}}}+\frac{D_1 \alpha_x}{\sqrt{E}}+\frac{D_2 \alpha_y}{\sqrt{G-\frac{F^2}{E}}}-\frac{F}{E}\frac{D_2 \alpha_x}{\sqrt{G-\frac{F^2}{E}}}) \nonumber\\
&& =K_\beta(D_1\cos(\alpha)+D_2\sin(\alpha))-K_\gamma(-D_1\sin(\alpha)+D_2\cos(\alpha))  \nonumber \\
&& \;+K^2(-\frac{f_{yx}-\frac{F}{E}f_{xx}-\frac{\partial }{\partial x}(\frac{F}{E})f_x}{\sqrt{E}\sqrt{G-\frac{F^2}{E}}}+\frac{(f_y-\frac{F}{E}f_x)\frac{\partial }{\partial x}(\sqrt{G-\frac{F^2}{E}})}{(G-\frac{F^2}{E})\sqrt{E}} \nonumber \\
&&\quad+\frac{f_{xy}}{\sqrt{G-\frac{F^2}{E}}\sqrt{E}}-\frac{\frac{\partial }{\partial y}(\sqrt{E})f_x}{\sqrt{G-\frac{F^2}{E}}E}-\frac{F}{E}\frac{f_{xx}}{\sqrt{G-\frac{F^2}{E}}\sqrt{E}} \nonumber \\
&&\quad\;+\frac{D_1 \alpha_x}{\sqrt{E}}+\frac{D_2 \alpha_y}{\sqrt{G-\frac{F^2}{E}}}-\frac{F}{E}\frac{D_2 \alpha_x}{\sqrt{G-\frac{F^2}{E}}}). 
\end{eqnarray*}
}
After some simplifications, we get a rephrasing on the compatibility condition
$(f_\gamma)_\beta=(f_\beta)_\gamma$ in the equivalent form
{\footnotesize
\begin{eqnarray}
&& 0= (f_\gamma)_\beta-(f_\beta)_\gamma= \left(K_\beta \cos(\alpha)+K_\gamma \sin(\alpha)\right)D_1\nonumber \\
&&\; +K^2\left(\frac{\frac{\partial }{\partial x}(\frac{F}{E})}{\sqrt{G-\frac{F^2}{E}}}-\frac{\frac{\partial }{\partial y}(\sqrt{E})}{\sqrt{E} \sqrt{G-\frac{F^2}{E}}}+ \frac{F}{E}\frac{\frac{\partial }{\partial x}(\sqrt{E})}{\sqrt{E} \sqrt{G-\frac{F^2}{E}}}+\frac{\alpha_x}{\sqrt{E}}\right)D_1 \nonumber \\
&&\quad +\left(K_\beta \sin(\alpha)-K_\gamma \cos(\alpha)+K^2\left( \frac{\frac{\partial }{\partial x}(\sqrt{G-\frac{F^2}{E}})}{\sqrt{G-\frac{F^2}{E}}\sqrt{E}}+\frac{\alpha_y-\frac{F}{E}\alpha_x}{\sqrt{G-\frac{F^2}{E}}} \right)\right)D_2\nonumber \\
&& =\left(K_\beta \cos(\alpha)+K_\gamma \sin(\alpha)+K^2\left(\frac{-E_y-\frac{F}{E}E_x+2F_x}{2E\sqrt{G-\frac{F^2}{E}}}+\frac{\alpha_x}{\sqrt{E}} \right)\right)D_1 \nonumber \\
&&\quad +\left(K_\beta \sin(\alpha)-K_\gamma \cos(\alpha)+K^2\left( \frac{\frac{F^2}{E^2}E_x-2\frac{F}{E}F_x+G_x}{2(G-\frac{F^2}{E})\sqrt{E}} +\frac{\alpha_y-\frac{F}{E}\alpha_x}{\sqrt{G-\frac{F^2}{E}}}\right)\right)D_2. \nonumber \\ 
\end{eqnarray}
}
Since the compatibility condition above has to be satisfied at every point of the domain and since $D_1$ and $D_2$ are linearly independent, the condition 4) is equivalent to the following PDE system
{\footnotesize
\begin{equation}
K_\beta \cos(\alpha)+K_\gamma \sin(\alpha)=K^2\left(-\frac{-E_y-\frac{F}{E}E_x+2F_x}{2E\sqrt{G-\frac{F^2}{E}}}-\frac{\alpha_x}{\sqrt{E}}\right), \label{kkg}
\end{equation}
\begin{equation}
K_\beta \sin(\alpha)-K_\gamma \cos(\alpha)=K^2\left( -\frac{\frac{F^2}{E^2}E_x-2\frac{F}{E}F_x+G_x}{2(G-\frac{F^2}{E})\sqrt{E}} -\frac{\alpha_y-\frac{F}{E}\alpha_x}{\sqrt{G-\frac{F^2}{E}}} \right). \label{kkb}
\end{equation}
}
Remark however, that the previous system \eqref{kkg}, \eqref{kkb} in $K_{\beta}$ and $K_{\gamma}$ is compatible and has a unique solution, namely the PDE system \eqref{condk} satisfied by $K$ in the statement of the theorem.
\end{enumerate}
\end{proof}

On the construction algorithm of the isothermic parameterization $(\beta, \gamma)$ starting from the initial, arbitrary parameterization $(x,y)$, it becomes natural to write down the actual formulas which describe this transformation of coordinates (change of chart). The Jacobian of this change of chart can be easily computed, and is presented in the following Proposition.

\begin{corollary}
The Jacobian of the change of chart $(\beta,\gamma) \mapsto (x,y)$ has the following entries:
{\small
\begin{eqnarray}
&& x_{\gamma}=K\left(\frac{\cos(\alpha)}{\sqrt{E}}-\frac{F}{\sqrt{E}}\frac{\sin(\alpha)}{\sqrt{EG-{F^2}}}\right), \label{xbetagamma} \\
&& x_{\beta}=K\left(\frac{-\sin(\alpha)}{\sqrt{E}}-\frac{F}{\sqrt{E}}\frac{\cos(\alpha)}{\sqrt{EG-{F^2}}}\right), \label{xbetagammaa}
\end{eqnarray}
}
and
{\small
\begin{eqnarray}
&& y_{\gamma}=K\left(\frac{\sin(\alpha)}{\sqrt{G-\frac{F^2}{E}}}\right), \label{ybetagamma} \\
&& y_{\beta}=K\left(\frac{\cos(\alpha)}{\sqrt{G-\frac{F^2}{E}}}\right). \label{ybetagammaa}
\end{eqnarray}
}
Moreover, each of the PDE systems $(x,y)$ listed above, namely \eqref{xbetagamma}-\eqref{xbetagammaa} and \eqref{ybetagamma}-\eqref{ybetagammaa}, has a unique solution, given a choice of initial values $(x_0,y_0)$.
\end{corollary}

\begin{proof}
It is straightforward to verify that the compatibility conditions, $x_{\gamma\beta}=x_{\beta\gamma}$ and $y_{\gamma\beta}=y_{\beta\gamma}$, are verified, because the condition 
$f_{\gamma\beta}=f_{\beta\gamma}$ is already fulfilled. Given the initial value $x(0,0)$ and $y(0,0)$ by results from \cite{Ivey}, the PDE system above has an unique solution represented 
by the chart $(x,y)$ which corresponds to $(\beta,\gamma)$.
\end{proof}

In the previous theorem, we proved that the existence of a certain scaling function $K(x,y)$ of arbitrarily given coordinates is necessary and sufficient in producing some new, isothermic coordinates. $K$ represents the solution of a system of differential equations, \eqref{condk}.
Moreover, we can provide a condition for such a scaling function $K$ to exist and be unique.

\begin{theorem}
As in the statement of the previous theorem, let $E,F,G$ and $l,m,n$ represent the coefficients of first and second fundamental forms of the arbitrary immersion $f(x,y)$, and let $\alpha$ be given by the same condition \eqref{alphann} as before. Assume the change of chart $(x,y)\mapsto(\beta, \gamma)$ is introduced in the same way as in the previous theorem. Let us now consider the PDE system as defined by \eqref{condk}.

The following condition
{\footnotesize
\begin{eqnarray} 
&& \left( \frac{\frac{\partial}{\partial y}\left(\frac{\frac{F^2}{E^2}E_x-2\frac{F}{E}F_x+G_x}{2(G-\frac{F^2}{E})\sqrt{E}}\right)-\frac{F}{E}\frac{\partial}{\partial x}\left(\frac{\frac{F^2}{E^2}E_x-2\frac{F}{E}F_x+G_x}{2(G-\frac{F^2}{E})\sqrt{E}}\right)}{\sqrt{G-\frac{F^2}{E}}}+\frac{\frac{\partial}{\partial x}\left(\frac{-E_y-\frac{F}{E}E_x+2F_x}{2E\sqrt{G-\frac{F^2}{E}}}\right)}{\sqrt{E}}+  \right. \nonumber  \\
&&\, \left(\frac{\frac{F^2}{E^2}E_x-2\frac{F}{E}F_x+G_x}{2(G-\frac{F^2}{E})\sqrt{E}}\right)\left(\frac{\alpha_x}{\sqrt{E}}\right)-\left(\frac{-E_y-\frac{F}{E}E_x+2F_x}{2E\sqrt{G-\frac{F^2}{E}}}\right)\left(\frac{\alpha_y-\frac{F}{E}\alpha_x}{\sqrt{G-\frac{F^2}{E}}}\right) \nonumber \\
&&\; -\left. \left(\frac{\frac{\partial}{\partial x}(\frac{\alpha_x}{\sqrt{E}})}{\sqrt{E}}+\left(\frac{\alpha_y-\frac{F}{E}\alpha_x}{\sqrt{G-\frac{F^2}{E}}}\right)\frac{\alpha_x}{\sqrt{E}}+\frac{\frac{\partial}{\partial y}(\frac{\alpha_y-\frac{F}{E}\alpha_x}{\sqrt{G-\frac{F^2}{E}}})}{\sqrt{G-\frac{F^2}{E}}}-\frac{F}{E}\frac{(\frac{\partial}{\partial x}(\frac{\alpha_y-\frac{F}{E}\alpha_x}{\sqrt{G-\frac{F^2}{E}}})+(\frac{\alpha_x}{\sqrt{E}})\alpha_x)}{\sqrt{G-\frac{F^2}{E}}}\right)   \right)=0, \nonumber \\ \label{NecSufCond}
\end{eqnarray}}
assures the existence of a function $K(x,y)$ that satisfies the system of differential equations \eqref{condk}. Moreover, if the solution $K$ exists, then it will be unique, given some appropriate initial value condition.

\end{theorem}

\

{\bf Important Remark}. The uniqueness of $K$ that satisfies the PDE system \eqref{condk} is natural, {\bf up to}  the following:
\begin{itemize}
 \item [a)]roto-translations (Euclidean motions) - which can be taken care of by setting a base point and an initial moving frame;
\item [b)]constant multiplications of $K$ (similarity transformations);
\item [c)]four possible values of $\alpha$ which all satisfy \eqref{alphann} (one for each quadrant: meaning that we will have to choose only one, when we perform the construction algorithm for isothermic coordinates).
\end{itemize}

\begin{proof}
Consider the system
{\footnotesize
\begin{eqnarray*}
&& K_\gamma=-K^2\left(\frac{\frac{\partial }{\partial x}(\frac{F}{E})}{\sqrt{G-\frac{F^2}{E}}}-\frac{\frac{\partial }{\partial y}(\sqrt{E})}{\sqrt{E} \sqrt{G-\frac{F^2}{E}}}+ \frac{F}{E}\frac{\frac{\partial }{\partial x}(\sqrt{E})}{\sqrt{E} \sqrt{G-\frac{F^2}{E}}}+\frac{\alpha_x}{\sqrt{E}}\right)\sin(\alpha) \nonumber \\
&&\; +K^2\left(\frac{\frac{\partial }{\partial x}(\sqrt{G-\frac{F^2}{E}})}{\sqrt{G-\frac{F^2}{E}}\sqrt{E}}+\frac{\alpha_y-\frac{F}{E}\alpha_x}{\sqrt{G-\frac{F^2}{E}}} \right)\cos(\alpha), \\
&& K_\beta=-K^2\left(\frac{\frac{\partial }{\partial x}(\frac{F}{E})}{\sqrt{G-\frac{F^2}{E}}}-\frac{\frac{\partial }{\partial y}(\sqrt{E})}{\sqrt{E} \sqrt{G-\frac{F^2}{E}}}+ \frac{F}{E}\frac{\frac{\partial }{\partial x}(\sqrt{E})}{\sqrt{E} \sqrt{G-\frac{F^2}{E}}}+\frac{\alpha_x}{\sqrt{E}}\right)\cos(\alpha) \nonumber \\
&&\; -K^2\left(\frac{\frac{\partial }{\partial x}(\sqrt{G-\frac{F^2}{E}})}{\sqrt{G-\frac{F^2}{E}}\sqrt{E}}+\frac{\alpha_y-\frac{F}{E}\alpha_x}{\sqrt{G-\frac{F^2}{E}}} \right)\sin(\alpha), 
\end{eqnarray*}
}
which, after simplifications, actually becomes equivalent to system \eqref{condk}-\eqref{condka}:
{\footnotesize
\begin{eqnarray*}
&& K_\gamma=K^2\left( \frac{\frac{F^2}{E^2}E_x-2\frac{F}{E}F_x+G_x}{2(G-\frac{F^2}{E})\sqrt{E}}\cos(\alpha)-\frac{-E_y-\frac{F}{E}E_x+2F_x}{2E\sqrt{G-\frac{F^2}{E}}}\sin(\alpha) \right) \nonumber\\
&&\; +K^2\left(-\frac{\alpha_x}{\sqrt{E}}\sin(\alpha)+\frac{\alpha_y-\frac{F}{E}\alpha_x}{\sqrt{G-\frac{F^2}{E}}}\cos(\alpha) \right), \nonumber \\
&& K_\beta=K^2\left( -\frac{\frac{F^2}{E^2}E_x-2\frac{F}{E}F_x+G_x}{2(G-\frac{F^2}{E})\sqrt{E}}\sin(\alpha)-\frac{-E_y-\frac{F}{E}E_x+2F_x}{2E\sqrt{G-\frac{F^2}{E}}}\cos(\alpha) \right) \nonumber \\
&&\; +K^2\left(-\frac{\alpha_x}{\sqrt{E}}\cos(\alpha)-\frac{\alpha_y-\frac{F}{E}\alpha_x}{\sqrt{G-\frac{F^2}{E}}}\sin(\alpha) \right). \nonumber
\end{eqnarray*}
}
Let
{\footnotesize 
\begin{eqnarray*}
&P_1(x,y)=\frac{\frac{F^2}{E^2}E_x-2\frac{F}{E}F_x+G_x}{2(G-\frac{F^2}{E})\sqrt{E}},
&Q_1(x,y)=\frac{\alpha_y-\frac{F}{E}\alpha_x}{\sqrt{G-\frac{F^2}{E}}},\\
&P_2(x,y)=\frac{-E_y-\frac{F}{E}E_x+2F_x}{2E\sqrt{G-\frac{F^2}{E}}},
&Q_2(x,y)=\frac{\alpha_x}{\sqrt{E}},
\end{eqnarray*}
}
in order to simplify the computation. With these simplifications, the system \eqref{condk}-\eqref{condka} becomes
{\footnotesize
\begin{eqnarray*}
&& K_\gamma=K^2\left( (P_1+Q_1) \cos(\alpha)-(P_2+Q_2) \sin(\alpha)\right), \\
&& K_\beta=K^2\left( -(P_1+Q_1) \sin(\alpha)-(P_2+Q_2) \cos(\alpha)\right).
\end{eqnarray*}
}
or in its equivalent form
{\footnotesize 
\begin{eqnarray*}
&& -\left(\frac{1}{K}\right)_\gamma=\frac{K_\gamma}{K^2}=\left( (P_1+Q_1) \cos(\alpha)-(P_2+Q_2) \sin(\alpha)\right), \\
&&  -\left(\frac{1}{K}\right)_\beta=\frac{K_\beta}{K^2}=\left( -(P_1+Q_1) \sin(\alpha)-(P_2+Q_2) \cos(\alpha)\right).
\end{eqnarray*}
}
Since the quantity {\footnotesize $K$} is always positive, 
showing compatibility condition {\footnotesize $K_{\gamma\beta}=K_{\beta\gamma}$} is equivalent to showing compatibility condition
{\footnotesize 
$$-\left(\frac{1}{K}\right)_{\gamma\beta}=-\left(\frac{1}{K}\right)_{\beta\gamma}.$$
}
Next, we write the mixed partial derivatives of {$\frac{1}{K}$} in the following way 
{\footnotesize
\begin{eqnarray*}
-\left(\frac{1}{K}\right)_{\gamma\beta}&=& 
K\left[\left( -\frac{\frac{\partial}{\partial x}P_1}{\sqrt{E}}-\frac{\frac{\partial}{\partial y}P_2-\frac{F}{E}\frac{\partial}{\partial x}P_2}{\sqrt{G-\frac{F^2}{E}}} \right) \sin(\alpha)\cos(\alpha) \right.\nonumber \\
&&+\left.\left(\frac{\frac{\partial}{\partial y}P_1-\frac{F}{E}\frac{\partial}{\partial x}P_1}{\sqrt{G-\frac{F^2}{E}}}\cos^2(\alpha)+\frac{\frac{\partial}{\partial x}P_2}{\sqrt{E}}\sin^2(\alpha) \right)\right]+W_1(\alpha),
\end{eqnarray*}
\begin{eqnarray*}
-\left(\frac{1}{K}\right)_{\beta\gamma}&=& 
K\left[\left( -\frac{\frac{\partial}{\partial x}P_1}{\sqrt{E}}-\frac{\frac{\partial}{\partial y}P_2-\frac{F}{E}\frac{\partial}{\partial x}P_2}{\sqrt{G-\frac{F^2}{E}}} \right) \sin(\alpha)\cos(\alpha) \right.\nonumber \\
&&+\left.\left(-\frac{\frac{\partial}{\partial y}P_1-\frac{F}{E}\frac{\partial}{\partial x}P_1}{\sqrt{G-\frac{F^2}{E}}}\sin^2(\alpha)-\frac{\frac{\partial}{\partial x}P_2}{\sqrt{E}}\cos^2(\alpha) \right)\right]+W_2(\alpha).
\end{eqnarray*}
}
In the above computations, we introduced the quantities
{\footnotesize
\begin{align*}
W_1(\alpha)= P_1 (\cos(\alpha))_\beta - P_2(\sin(\alpha))_\beta + \left(Q_1 \cos(\alpha)-Q_2 \sin(\alpha)\right)_\beta,\\
W_2(\alpha)= -P_1 (\sin(\alpha))_\gamma - P_2(\cos(\alpha))_\gamma + \left(-Q_1 \sin(\alpha)-Q_2 \cos(\alpha)\right)_\gamma,
\end{align*}
}
for simplification. The expression {\footnotesize $W_1(\alpha)-W_2(\alpha)$} is evaluated below:
{\footnotesize
\begin{eqnarray*}
&&  W_1(\alpha)-W_2(\alpha)=K\left( -P_1\left(-\frac{\alpha_x}{\sqrt{E}}\sin^2(\alpha)+\frac{\alpha_y-\frac{F}{E}\alpha_x}{\sqrt{G-\frac{F^2}{E}}}\sin(\alpha)\cos(\alpha)\right)\right) \nonumber \\
&& -K\left(P_2\left(-\frac{\alpha_x}{\sqrt{E}}\sin(\alpha)\cos(\alpha)+\frac{\alpha_y-\frac{F}{E}\alpha_x}{\sqrt{G-\frac{F^2}{E}}}\cos^2(\alpha)\right) \right) 
\nonumber \\ && 
+K\left(P_1\left(\frac{\alpha_x}{\sqrt{E}}\cos^2(\alpha)+\frac{\alpha_y-\frac{F}{E}\alpha_x}{\sqrt{G-\frac{F^2}{E}}}\sin(\alpha)\cos(\alpha)\right) \right) \nonumber \\
&& -K\left(P_2\left(\frac{\alpha_x}{\sqrt{E}}\sin(\alpha)\cos(\alpha)+\frac{\alpha_y-\frac{F}{E}\alpha_x}{\sqrt{G-\frac{F^2}{E}}}\sin^2(\alpha)\right) \right) \nonumber \\
&& -K\left( \frac{\frac{\partial}{\partial x}(\frac{\alpha_x}{\sqrt{E}})}{\sqrt{E}}\cos^2(\alpha)+\frac{\frac{\partial}{\partial x}(\frac{\alpha_y-\frac{F}{E}\alpha_x}{\sqrt{G-\frac{F^2}{E}}})}{\sqrt{E}}\sin(\alpha)\cos(\alpha)-\frac{\alpha^2_x}{E}\sin(\alpha)\cos(\alpha)  \right) \nonumber \\
&& -K\left( \frac{\frac{\partial}{\partial y}(\frac{\alpha_x}{\sqrt{E}})}{\sqrt{G-\frac{F^2}{E}}}\sin(\alpha)\cos(\alpha)+\frac{\frac{\partial}{\partial y}(\frac{\alpha_y-\frac{F}{E}\alpha_x}{\sqrt{G-\frac{F^2}{E}}})}{\sqrt{G-\frac{F^2}{E}}}\sin^2(\alpha)-\frac{\alpha^2_x}{\sqrt{E}\sqrt{G-\frac{F^2}{E}}}\sin(\alpha)\cos(\alpha) \right) \nonumber \\
&& -K\left(-\frac{F}{E}\frac{(\frac{\partial}{\partial x}(\frac{\alpha_y-\frac{F}{E}\alpha_x}{\sqrt{G-\frac{F^2}{E}}})+(\frac{\alpha_x}{\sqrt{E}})\alpha_x)}{\sqrt{G-\frac{F^2}{E}}}\sin^2(\alpha)+(\frac{\alpha_y-\frac{F}{E}\alpha_x}{\sqrt{G-\frac{F^2}{E}}})\frac{\alpha_x}{\sqrt{E}}\cos^2(\alpha)  \right) \nonumber \\
&& -K\left( \frac{\frac{\partial}{\partial x}(\frac{\alpha_x}{\sqrt{E}})}{\sqrt{E}}\sin^2(\alpha)-\frac{\frac{\partial}{\partial x}(\frac{\alpha_y-\frac{F}{E}\alpha_x}{\sqrt{G-\frac{F^2}{E}}})}{\sqrt{E}}\sin(\alpha)\cos(\alpha)+\frac{\alpha^2_x}{E}\sin(\alpha)\cos(\alpha)  \right) \nonumber \\
&& -K\left(-\frac{\frac{\partial}{\partial y}(\frac{\alpha_x}{\sqrt{E}})}{\sqrt{G-\frac{F^2}{E}}}\sin(\alpha)\cos(\alpha)+\frac{\frac{\partial}{\partial y}(\frac{\alpha_y-\frac{F}{E}\alpha_x}{\sqrt{G-\frac{F^2}{E}}})}{\sqrt{G-\frac{F^2}{E}}}\cos^2(\alpha)+\frac{\alpha^2_x}{\sqrt{E}\sqrt{G-\frac{F^2}{E}}}\sin(\alpha)\cos(\alpha) \right) \nonumber \\
&& -K\left(-\frac{F}{E}\frac{(\frac{\partial}{\partial x}(\frac{\alpha_y-\frac{F}{E}\alpha_x}{\sqrt{G-\frac{F^2}{E}}})+(\frac{\alpha_x}{\sqrt{E}})\alpha_x)}{\sqrt{G-\frac{F^2}{E}}}\cos^2(\alpha)+(\frac{\alpha_y-\frac{F}{E}\alpha_x}{\sqrt{G-\frac{F^2}{E}}})\frac{\alpha_x}{\sqrt{E}}\sin^2(\alpha)  \right)=\nonumber\\
&&K\left[\left( P_1(\frac{\alpha_x}{\sqrt{E}})-P_2(\frac{\alpha_y-\frac{F}{E}\alpha_x}{\sqrt{G-\frac{F^2}{E}}})\right) \nonumber 
 -\left( \frac{\frac{\partial}{\partial x}(\frac{\alpha_x}{\sqrt{E}})}{\sqrt{E}}+(\frac{\alpha_y-\frac{F}{E}\alpha_x}{\sqrt{G-\frac{F^2}{E}}})\frac{\alpha_x}{\sqrt{E}}+\frac{\frac{\partial}{\partial y}(\frac{\alpha_y-\frac{F}{E}\alpha_x}{\sqrt{G-\frac{F^2}{E}}})}{\sqrt{G-\frac{F^2}{E}}}  \right)\right. \nonumber \\
&& \quad\; \left.-\left(-\frac{F}{E}\frac{(\frac{\partial}{\partial x}(\frac{\alpha_y-\frac{F}{E}\alpha_x}{\sqrt{G-\frac{F^2}{E}}})+(\frac{\alpha_x}{\sqrt{E}})\alpha_x)}{\sqrt{G-\frac{F^2}{E}}}  \right)\right]. \nonumber
\end{eqnarray*}}
Next, we impose the compatibility condition
{\footnotesize
\begin{eqnarray*}0&=&\left(\frac{1}{K}\right)_{\beta\gamma}-\left(\frac{1}{K}\right)_{\gamma\beta}
=K\left( \frac{\frac{\partial}{\partial y}P_1-\frac{F}{E}\frac{\partial}{\partial x}P_1}{\sqrt{G-\frac{F^2}{E}}}+\frac{\frac{\partial}{\partial x}P_2}{\sqrt{E}} \right)+W_1(\alpha)-W_2(\alpha)=\nonumber\\
&=& K\left( \frac{\frac{\partial}{\partial y}P_1-\frac{F}{E}\frac{\partial}{\partial x}P_1}{\sqrt{G-\frac{F^2}{E}}}+\frac{\frac{\partial}{\partial x}P_2}{\sqrt{E}}+ P_1(\frac{\alpha_x}{\sqrt{E}})-P_2(\frac{\alpha_y-\frac{F}{E}\alpha_x}{\sqrt{G-\frac{F^2}{E}}}) \right) \nonumber  \\
&& -K\left( \frac{\frac{\partial}{\partial x}(\frac{\alpha_x}{\sqrt{E}})}{\sqrt{E}}+(\frac{\alpha_y-\frac{F}{E}\alpha_x}{\sqrt{G-\frac{F^2}{E}}})\frac{\alpha_x}{\sqrt{E}}+\frac{\frac{\partial}{\partial y}(\frac{\alpha_y-\frac{F}{E}\alpha_x}{\sqrt{G-\frac{F^2}{E}}})}{\sqrt{G-\frac{F^2}{E}}}-\frac{F}{E}\frac{(\frac{\partial}{\partial x}(\frac{\alpha_y-\frac{F}{E}\alpha_x}{\sqrt{G-\frac{F^2}{E}}})+(\frac{\alpha_x}{\sqrt{E}})\alpha_x)}{\sqrt{G-\frac{F^2}{E}}}   \right). \nonumber \\
\end{eqnarray*}
}
After dividing the equation by $K$ and replacing the expressions of $P_1$ and $P_2$, we obtain the condition stated in the theorem.
\end{proof}

\begin{corollary}
In case $F=0$, the necessary and sufficient condition \eqref{NecSufCond} reduces to
{\footnotesize
\begin{eqnarray*}
&&\frac{1}{\sqrt{G}}\frac{\partial}{\partial y}\left(\frac{G_x}{2 G \sqrt{E}}\right)-
\frac{1}{\sqrt{E}}\frac{\partial}{\partial x}\left(\frac{E_y}{2 \sqrt{G} E}\right)\nonumber\\
&&\qquad+\frac{G_x \alpha_x+E_y \alpha_y}{2 G E}-
\left(\frac{\frac{\partial}{\partial x}(\frac{\alpha_x}{\sqrt{E}})}{\sqrt{E}}+
      \frac{\alpha_y \alpha_x}{\sqrt{G E}}+
      \frac{\frac{\partial}{\partial y}(\frac{\alpha_y}{\sqrt{G}})}{\sqrt{G}}
\right)=0. \nonumber 
\end{eqnarray*}}
\end{corollary}
\begin{corollary}
For isothermal coordinate lines, $G=E$ and $F=0$, the necessary and sufficient condition \eqref{NecSufCond} reduces to
{\footnotesize
\begin{equation*}
\frac{\partial }{\partial x}\left(\frac{\alpha_x}{E}\right)+
\frac{\alpha_y \alpha_x}{E}+
\frac{\partial }{\partial y}\left(\frac{\alpha_y}{E}\right)
=0. \nonumber 
\end{equation*}}
\end{corollary}
\begin{corollary}
In case $F=0$, $m=0$ and away from umbilic points, the rotation angle $\alpha$ is zero and the necessary and sufficient condition \eqref{NecSufCond} reduces to
{\footnotesize
\begin{eqnarray*}
&& \left[\ln{\left(\frac{G}{E}\right)}\right]_{xy}=0\nonumber, 
\end{eqnarray*}}
or in its equivalent form
{\small
\begin{eqnarray*}
&& \frac{G}{E}=\phi(x)\psi(y)\nonumber, 
\end{eqnarray*}}
for any arbitral smooth functions $\phi(x)$ and $\psi(y)$.
\end{corollary}

Finally, we are resuming the construction of the isothermic coordinates, based on the algorithm
we presented starting from an initial coordinate chart $(x,y)$, followed by a well-chosen rotation of angle $\alpha(x,y)$ and a re-scaling $K(x,y)$, all up to roto-translations and similarity transformations.

\begin{example}
How is the actual construction algorithm handled numerically? In practice, we integrate along $\gamma$ and $\beta$ and draw a virtual mesh to construct the surfaces:
\begin{equation}
 f(\gamma,\beta)=f(0,0)+\int_0^\beta \!f_{\gamma}(0,\bar{\beta}(x,y)) \,d\bar{\beta}+\int_0^\gamma \!f_{\beta}(\bar{\gamma}(x,y),\beta(x,y)) \,d\bar{\gamma}. \label{totalODE}
\end{equation}
We start constructing the surface in a neighborhood of $(0,0)$ by solving a succession of PDE's.
First, we fix $\gamma=0$ and move along $\beta$, then we fix $\beta$, and move along $\gamma$. Because of
the compatibility condition, reversing the order of integration would not change the value $f(\gamma,\beta)$.

Note that in \eqref{totalODE} we integrate along $(\gamma,\beta)$, 
while the integrands are given in terms of the initial parameters $(x,y)$.
In order to integrate, we need first to solve and keep track of all the
quantities 
$$x(\gamma,\beta),\;y(\gamma,\beta),\;f_{\gamma}(x(\gamma,\beta),y(\gamma,\beta)), 
\mbox{ and } f_{\beta}(x(\gamma,\beta),y(\gamma,\beta)),$$ 
and with these solve for $f(\gamma,\beta)$ in \eqref{totalODE}.
Thus, the problem is solved in two steps:
\begin{enumerate}
\item First, constructing the mapping $x(\gamma,\beta)$, $y(\gamma,\beta)$ and the 
vector fields $f_{\gamma}$, $f_{\beta}$ by solving numerically, with a fourth-order Runge-Kutta method,
the systems \eqref{xbetagamma}-\eqref{ybetagammaa}, and \eqref{ffgambet}-\eqref{ffgambeta}, respectively.
\item Then, constructing the surface locally by solving numerically the system \eqref{totalODE} 
with a fourth-order Runge-Kutta method.
\end{enumerate}

The Unduloid is an example of a CMC-Delaunay surface which is generated by revolving a curve in arclength $s$, 
namely $C=(X(s), Z(s))$ in the $xz$-plane, around the $z$-axis. This surface of revolution is therefore parameterized as

$$ f(s,\theta)= \left( \begin{array}{ccc}
X(s)\cos{\theta}   \\
X(s)\sin{\theta}   \\
Z(s)   \end{array} \right),$$

where the angle $\theta$ takes values between $0$ and $2 \pi$.

Note that the parameters $(s, \theta)$ are not isothermal.  

We applied our algorithm starting from the parameterization $(s, \theta)$ (in the role of original chart $(x,y)$), and 
constructed a pair of isothermic coordinates, $(\beta, \gamma)$.

The figure below represents the CMC unduloid that we obtained by applying our construction algorithm.  

\begin{figure}[!h]
\begin{center}
\includegraphics[scale=0.5, angle=90]{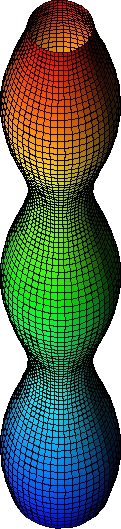}  
\end{center}
\caption{Unduloid surface in isothermic parametrization.}
\end{figure}

\end{example}

{\bf Acknowledgement.} The authors are grateful for the financial support received from the National Science Foundation as grant NSF-DMS 0908177, which was very helpful for their joint research in this area. 

\end{document}